\documentclass[11pt,reqno,section]{amsart}
\usepackage[english]{babel}
\usepackage[utf8]{inputenc}
\usepackage{amsmath}
\usepackage{amsthm}
\usepackage{verbatim}
\usepackage{enumerate}
\usepackage{amssymb}
\usepackage{amsfonts}
\usepackage{amscd,bezier}
\usepackage{anysize}
\usepackage{indentfirst}
\usepackage{upref}
\usepackage{color}
\usepackage[pagewise]{lineno}
\usepackage{tikz}
\usepackage{theoremref}
\usetikzlibrary{matrix}
\usepackage{pstricks}
\usepackage{subfig}
\usepackage{graphicx}
\usepackage{pstricks}
\usepackage{amscd}
\usepackage{relsize}
\usepackage[toc,page]{appendix}
\usepackage{commath}
\usepackage{mathabx}
\usepackage{MnSymbol}
\usepackage{xcolor}

\allowdisplaybreaks

\numberwithin{equation}{section}

\newtheorem{theorem}{Theorem}[section]
\newtheorem{lemma}[theorem]{Lemma}

\newtheorem{definition}[theorem]{Definition}
\newtheorem{proposition}[theorem]{Proposition}
\newtheorem{corollary}[theorem]{Corollary}

\newtheorem{example}[theorem]{Example}

\newtheorem{remark}[theorem]{Remark}

\marginsize{3cm}{3cm}{3cm}{3cm}

\newcommand{\walpha}{{\widetilde{\alpha}}}
\newcommand{\wbeta}{{\widetilde{\beta}}}

\newcommand{\muD}{{\mu\mathrm{D}}}
\newcommand{\cD}{{\ceg\mathrm{D}}}

\newcommand{\eD}{{\exp\mathrm{D}}}
\newcommand{\omegaD}{{\omega\mathrm{D}}}

\newcommand{\Si}{\Sigma}
\newcommand{\ED}{\mathrm{ED}}
\newcommand{\Id}{\mathrm{Id}}

\newcommand{\im}{\mathrm{im}}

\renewcommand{\P}{\mathrm{P}}

\newcommand{\sgn}{\mathrm{sgn}}

\newcommand{\qeg}{\mathrm{q}}
\newcommand{\ceg}{\mathrm{c}}

\newcommand{\x}{\mathrm{x}}
\newcommand{\A}{\mathrm{A}}

\numberwithin{equation}{section}

\allowdisplaybreaks
\usepackage{color}

\title[Growth rate comparisons]{The interplay of $\mu$-dichotomy, bounded growth, and spectral properties via growth rate comparisons.}

\author[N. Jara]{N\'estor Jara$^{*}$}

\thanks{$^{*}$
	This author was partially supported by ANID, Beca de Doctorado Nacional 21220105.}

\author[C. A. Gallegos]{Claudio A. Gallegos}
\address{Universidad de Chile, UCH, Facultad de Ciencias, Departamento de Matem\'aticas, Casilla 653, Santiago, Chile.}
\email{claudiogallegos@uchile.cl, nestor.jara@ug.uchile.cl} 

\date{}

\begin{document}
	
	\begin{abstract}
  We investigate the behavior of the dichotomy spectrum of nonautonomous linear systems under general growth rates. By introducing comparison criteria we clarify how $\mu$-dichotomy and $\mu$-bounded growth interact. We also study the evolution of the dichotomy spectrum under these comparisons, revealing that faster growth rates compress the spectrum, while slower growth rates expand it. Moreover, we introduce equivalence relations on the set of growth rates, which enable us to establish that, for any given system, there exists --up to equivalence-- at most one growth rate under which both properties, bounded growth and dichotomy, hold. Finally, we show that these equivalence relations lead to a classification of dichotomy spectra. Our results are valid in both discrete and continuous time settings.
	\end{abstract}
	

	

	\maketitle	
	
	
\section{Introduction}

To provide context for the forthcoming discussion, consider the following nonautonomous linear system:
\begin{equation}\label{1}
 \dot{\x}=2|t|\x(t), \qquad t\in \mathbb{R}. 
\end{equation}

The exponential dichotomy spectrum of this system --also known as the Sacker $\&$ Sell spectrum \cite{Sacker}-- is equal to $\{+\infty\}$ (following the convention from \cite{Rasmussen,Rasmussen2}, see Remark \ref{910}). Namely, regardless of the shift applied to the system, it always exhibits an exponential dichotomy with projector $\P \equiv 0$.

The preceding remark places system \eqref{1}, along with others exhibiting similar behavior, outside the scope of many results that rely on conditions involving the dichotomy spectrum. For instance, results concerning normal forms \cite{CJ, Siegmund3} and smooth linearization \cite{Cuong} are based on the so-called nonresonance condition, which provides no information when the spectrum is equal to $\{+\infty\}$, $\{-\infty\}$ or $\{\pm\infty\}$. 

Recent research has focused on the concept of 
$\mu$-dichotomy, which offers a framework for analyzing dichotomies defined by growth rates. These growth rates are positive functions $\mu$ that emulate exponential behavior, thereby enabling a significantly broader setting. This notion of dichotomy naturally gives rise to a corresponding concept of dichotomy spectrum (see Def.~\ref{dichspectrum}), which encompasses the classical Sacker $\&$ Sell spectrum as a particular case.

Returning to system \eqref{1}, a closer examination reveals that, although it satisfies exponential dichotomy, it also exhibits a dichotomy governed by the growth rate defined by $t \mapsto \qeg(t) = e^{\sgn(t) \cdot t^2}$, for all $t\in\mathbb{R}$, which we refer to as the {\it quadratic exponential rate} (see Def.~\ref{DefNmuD}). Moreover, when computing the quadratic exponential dichotomy spectrum of system \eqref{1}, we observe that it corresponds to $\{1\}$, in contrast to the Sacker $\&$ Sell spectrum which, as mentioned above, is $\{+\infty\}$. This observation illustrates how adopting a different growth rate can provide a more refined description of a system's behavior.

In addition, it is straightforward to verify that system \eqref{1} does not exhibit exponential bounded growth (see, for instance, \cite[Sect.~3.1]{Siegmund2002} or \cite{Coppel}). However, it does exhibit quadratic exponential bounded growth (see Def.~\ref{NmuGrowth}), since its evolution operator (a precise definition is given in Sect.~\ref{dichotomies}) is defined as $\Psi(t,s) = \frac{\qeg(t)}{\qeg(s)}$ for all $t, s \in \mathbb{R}$.

Overall, the exponential rate fails to adequately describe the dynamics arising from \eqref{1}. However, by appropriately choosing the growth rate --in this case, the quadratic exponential-- the system's behavior becomes significantly more tractable.

It is worth noting that, although system \eqref{1} is defined in continuous time, we may consider its discrete analogue given by the difference equation $x(k+1) = A(k)x(k)$, where
\begin{equation*}
	A(k) = \left\{
	\begin{array}{lcc}
		e^{2n+1} & \text{if} & k \geqslant 0, \\
		e^{-2n-1} & \text{if} & k < 0.
	\end{array}
	\right.
\end{equation*}
 One can verify that its evolution operator in the discrete sense (see \eqref{evo}) is again given by $\Phi(k,n) = \frac{\qeg(k)}{\qeg(n)}$ for all $k, n \in \mathbb{Z}$. In this sense, the difference equation serves as a discrete counterpart to the system \eqref{1}. Therefore, the observations made in the continuous setting for \eqref{1} also hold in this discrete context.


\subsection{Novelty and Main Results}

In this article, we generalize the observations made in the initial example and establish a general criterion for determining when a system has a dichotomy spectrum of the form $\{+\infty\}$, $\{-\infty\}$, or $\{\pm\infty\}$. More specifically, we introduce several notions of comparison for growth rates and analyze, in each case, how these new definitions contribute to a deeper understanding of the behavior of the dichotomy spectrum. Furthermore, we explore several consequences derived from the new comparison framework introduced in this work, which we summarize as follows:

\begin{itemize}
\item In Sect.~\ref{section3}, we introduce the first comparison criterion for growth rates $\mu,\omega\colon\mathbb{Z}\to\mathbb{R}^{+}$, where we define what it means for $\mu$ to be faster than $\omega$, denoted by $\mu \gg \omega$ (see Def.~\ref{901} for the formal statement). This notion underlies the main findings of the section, which we highlight as follows:
\medskip
\begin{enumerate}
 \item[$\mathsection$] If a system has $\mu$-dichotomy, then it cannot exhibit $\omega$-bounded growth for any rate $\omega$ slower than $\mu$; see Thm.~\ref{805} and Cor.~\ref{721}.
    \medskip
\item[$\mathsection$] If a system exhibits $\omega$-bounded growth, then it cannot admit $\mu$-dichotomy for any rate $\mu$ faster than $\omega$; see Thm.~\ref{806} and Cor.~\ref{722}.    
\end{enumerate}

\medskip

\item In Sect.~\ref{section4}, we introduce a weaker comparison between growth rates $\omega$ and $\mu$, denoted by $\mu \succ \omega$, which we refer to as $\mu$ being weakly faster than $\omega$ (see Def.~\ref{902} for the formal statement). Roughly speaking, in this section we study how the dichotomy spectra associated with different growth rates evolve under weak comparison. This analysis leads to the main results, Theorems~\ref{808} and \ref{809}, which --combined with the results of Sect.~\ref{section3}-- provide a clear understanding of the spectral behavior: when growth rates become faster (even in the weak sense), the spectrum concentrates near zero; conversely, when the rates become slower (or weakly slower), the spectrum expands in both directions, approaching $-\infty$ and $+\infty$.

\medskip

\item Consider a discrete --or continuous-- time framework for a nonautonomous linear system. Let $\mathcal{G}$ be the set of all growth rates $\mu$ in this time setting. In Sect.\ref{section5}, we introduce equivalence relations on $\mathcal{G}$ (see Def.\ref{equivalences}) with the goal of clarifying the duality between $\mu$-dichotomy and $\mu$-bounded growth. Specifically, we establish that, up to a suitable notion of equivalence, a given system admits at most one growth rate $\mu$ under which it simultaneously exhibits both $\mu$-bounded growth and $\mu$-dichotomy; see Thm.~\ref{811}. Moreover, for each equivalence relation defined on $\mathcal{G}$, we characterize the dichotomy spectrum in every corresponding equivalence class; see Thm.~\ref{908}. 

\end{itemize}

We emphasize that, although most of our definitions and results are developed in the discrete-time setting, the underlying concepts extend naturally to continuous-time systems. Accordingly, our contributions are valid in both settings.

\section{Preliminaries}
In this section we present the main preliminary definitions we need to describe our results.

\subsection{Dichotomies with growth rates}\label{dichotomies}
Consider the nonautonomous difference equation
\begin{equation}\label{700}
x(k+1)=A(k)x(k),\qquad k\in \mathbb{Z},
\end{equation}
and suppose that $A:\mathbb{Z}\to \mathcal{M}_d(\mathbb{R})$ is nonsingular. The \textbf{evolution operator} for \eqref{700} is the map $\Phi: \mathbb{Z} \times \mathbb{Z} \to GL_d(\mathbb{R})$, defined by
\begin{equation}\label{evo}
    \Phi(k,n)= \left\{ \begin{array}{lc}
             A(k-1)A(k-2)\cdots A(n) &  \text{if}\,\,k > n, \\
              \Id & \text{ if}\,\,k=n, \\
              A^{-1}(k)A^{-1}(k+1)\cdots A^{-1}(n-1) & \text{if}\,\,k<n.
             \end{array}
   \right.
   \end{equation}
Note that the evolution operator has the property that the map $k\mapsto \Phi(k,n)$ is a solution to the matrix-valued Cauchy problem:
\begin{equation*}
    \left\{ \begin{array}{lcc}
             X(k+1) & = &A(k)X(k), \\
             \\ X(n) &=&\Id.
             \end{array}
             \right.
\end{equation*}

For dynamics in continuous time, that is, arising from a nonautonomous differential equation
\begin{equation}\label{3}
    \dot{\x}=\A(t)\x(t),\qquad t\in\mathbb{R},
\end{equation}
where $\A:\mathbb{R}\to \mathcal{M}_d$ is locally integrable, an analogous definition is given for its evolution operator $\Psi:\mathbb{R}\times\mathbb{R}\to GL_d(\mathbb{R})$, as a solution to the corresponding matrix-valued Cauchy problem. Although most of the manuscript is written in terms of discrete-time dynamics, the concepts presented are equally applicable to the continuous-time framework; thus, it is convenient to keep such systems in consideration.

An \textbf{invariant projector} for \eqref{700} is a map $\P:\mathbb{Z}\to \mathcal{M}_d(\mathbb{R})$ of idempotents such that 
\begin{equation*}
    A(k) \P(k)=\P(k+1)A(k),\quad\,\forall\,k\in \mathbb{Z}.
\end{equation*}
Note that this condition implies that each $\P(n)$ has the same rank, for $n\in\mathbb{Z}$. As a consequence, if $n\mapsto\P(n)$ is an invariant projector for \eqref{700}, then
\begin{equation}\label{2}
  \P(k)\Phi(k,n)=\Phi(k,n)\P(n),\qquad\forall\,n,k\in \mathbb{Z}.
\end{equation}
Through the expression \eqref{2}, but replacing $\Phi$ with $\Psi$, an analogous definition is given for the invariant projector associated to the differential system \eqref{3}.

The following definition is a crucial tool for describing behaviors more general than exponential ones. This concept is based on a similar notion presented in \cite{Silva discreto} for dynamics defined on $\mathbb{Z}^+$, and is also based on a similar concept presented in \cite{Silva} for continuous time flows. This notion has gained importance recently, as it has been employed on different works \cite{CJ,GJ}. Nonetheless, this concept has an already long trajectory in literature. To the best of our knowledge, a preliminary notion of it first appeared in the works of Pinto and Naul\'in \cite{PN2,PN1,PN3} and has been applied to describe different scenarios \cite{Bento, ZFY}.

\begin{definition}\label{growthrate}
    We say that a map $\mu:\mathbb{Z},\mathbb{R}\to \mathbb{R}^+$ is a \textbf{growth rate} if it is non-decreasing and verifies $\mu(0)=1$,  $\lim_{n\to +\infty}\mu(n)=+\infty$ and $\lim_{n\to -\infty}\mu(n)=0$. 
\end{definition}

If the domain of $\mu$ is $\mathbb{Z}$ we refer to it as a discrete growth rate, while if its domain is $\mathbb{R}$ we refer to it as a continuous growth rate. In the case of a continuous growth rate, we will be interested mainly in differentiable growth rates.

In order to give some examples, we employ the {\bf sign map} $\sgn:\mathbb{Z},\mathbb{R}\to \mathbb{R}$ given by $\sgn(n)=1$ if $n> 0$, $\sgn(0)=0$ and $\sgn(n)=-1$ if $n<0$.

For instance, the map $n\mapsto \exp(n)= e^n$ defines the \textbf{exponential growth rate}, while $p:\mathbb{Z}\to\mathbb{R}^+$ given by $p(n)=n^{\sgn(n)}$ for $n\neq 0$ and $p(0)=1$ defines the \textbf{polynomial growth rate}. This last notion of growth has been employed on \cite{Dragicevic6} for dynamics defined on $\mathbb{Z}^+$ and a similar notion was used on \cite[Ex.~3.3]{CJ} in the context of continuous time dynamics, although in that framework the definition is usually presented as $p(t)=(1+|t|)^{\sgn(t)}$, for all $t\in\mathbb{R}$.

Other important examples that we will employ throughout the text are: the map $\qeg:\mathbb{Z}\to \mathbb{R}^+$ given by $\qeg(n)=e^{\sgn(n)n^2}$, which we call the \textbf{quadratic exponential growth rate}, and the map $\ceg:\mathbb{Z}\to \mathbb{R}^+$, given by $\ceg(n)=e^{n^3}$, which we call the \textbf{cubic exponential growth rate}. The cubic exponential growth rate was employed on \cite[Ex.~3.6]{CJ} in the context of continuous time dynamics.

 Note that if $\hat{\mu}:[0,+\infty)\to \mathbb{R}^+$ is an increasing function with $\hat{\mu}(0)=1$ and $\displaystyle\lim_{t\to\infty}\hat{\mu}(t)=+\infty$, then defining $\mu:\mathbb{R}\to \mathbb{R}^+$ by
    \begin{equation*}
		\mu(t):= \left\{ \begin{array}{lcc}
		\hat{\mu}(t) &  \text{ if } &   t\geqslant 0, \\
		 \hat{\mu}(-t)^{-1} & \text{ if }& t< 0,
		\end{array}
		\right.
		\end{equation*}
        it turns out to be a growth rate.
 
For similar purposes, in \cite[p. 345]{Potzsche} the author introduces the concept of {\it generalized exponential function}, defined as a map $e_a(k, n)$ depending on two discrete variables. While this tool is more general, its primary application arises when $e_a(k, n) = \mu(k) / \mu(n)$ for some growth rate $\mu$.

\begin{definition}\label{DefNmuD}
    Consider a discrete growth rate $\mu$. We say that system \eqref{700} admits a \textbf{$\mu$-dichotomy} $(\muD)$ on $\mathbb{Z}$ if there is an invariant projector $n\mapsto \P(n)$ and constants $K\geqslant 1$, $\alpha<0$, $\beta>0$ such that 
      \[\left\{ \begin{array}{lc}
            \norm{\Phi(k,n)\P(n)}\leqslant K\left(\dfrac{\mu(k)}{\mu(n)}\right)^\alpha, &\forall\,\, k \geqslant n, \\
             \norm{\Phi(k,n)[\Id-\P(n)]}\leqslant K\left(\dfrac{\mu(k)}{\mu(n)}\right)^\beta , &\forall\,\, k\leqslant n.
             \end{array}
   \right.\]
\end{definition}

The above concept is similarly defined for continuous-time dynamics, such as those arising from equation \eqref{3}, by considering a continuous growth rate and replacing $\Phi$ with $\Psi$.  For more details, we refer the reader to \cite[Def. 2.3]{CJ}, \cite[p. 621]{Silva} and \cite[Def. 2.2]{GJ}.

The above estimates are referred to as a dichotomy with \textbf{parameters} $(\P; \alpha, \beta)$ and constant $K$. Note that if $\P = \Id$, then there is no parameter $\beta$, while if $\P = 0$, there is no parameter $\alpha$. In both cases, we retain the notation but replace the missing parameter with an asterisk ($*$).

The notation used in this paper ensures a consistent unification of terminology for both difference and differential equations, aligning with the notation of recent works that investigate the spectral problem \cite{GJ,Silva}. The concept of $\mu$-dichotomy encompasses also other previously studied notions of dichotomy, as the exponential or polynomial, defined respectively by those growth rates.

Notably, the concept of $\mu$-dichotomy has also, and mainly, emerged in the context of nonuniform dichotomies \cite{Silva discreto,Silva}, that is, dichotomies in which the convergence depends not only on the relation of the times considered on the evolution operator, but also on the initial condition itself. For the purposes of this paper, we leave out these kind of dynamics.

\subsection{Dichotomy spectrum and bounded growth}

An important auxiliary concept we require is that of a weighted system, which serves as a tool for modifying systems in order to explore their spectra.

\begin{definition}
Let $\mu$ be a discrete growth rate. Given a real number $\gamma$, we define the \textbf{{\rm ($\mu,\gamma$)}-weighted} system (or simply $\gamma$-weighted system, if the growth rate is clear) associated to \eqref{700} as
\begin{equation}\label{701}
    x(k+1)=A(k)\left(\frac{\mu(k+1)}{\mu(k)}\right)^{-\gamma}x(k),\qquad k\in \mathbb{Z}.
\end{equation}
\end{definition}

Note that the evolution operator for the $(\mu,\gamma)$-weighted system is given by
\begin{equation}\label{761}
    \Phi_{\mu,\gamma}(k,n)=\left(\frac{\mu(k)}{\mu(n)}\right)^{-\gamma}\Phi(k,n),\qquad k,n\in \mathbb{Z},
\end{equation}
where $\Phi(k,n)$ is the evolution operator of the system \eqref{700}. If the growth rate $\mu$ is clear, we just denote it by $\Phi_\gamma$.

\begin{remark}
   {\rm  In the case of continuous time dynamics, we instead consider the $(\mu,\gamma)$-\textbf{shifted} system from \eqref{3}, which is given by. 
\begin{equation}\label{shift}
    \dot{\x}=\left[\A(t)-\gamma\frac{\dot{\mu}(t)}{\mu(t)}\Id\right]\x(t),
\end{equation}
where in this case $\mu:\mathbb{R}\to \mathbb{R}^+$ is a differentiable growth rate. The remarkable property of this shifted differential equation is that its evolution operator is given in the same fashion as \eqref{761}, but replacing $\Phi$ with $\Psi$ and considering continuous-time variables.}
\end{remark}

\begin{remark}\label{910}
{\rm
Following the convention introduced by M. Rasmussen \cite{Rasmussen,Rasmussen2} we will state that \eqref{701} has $\mu$-dichotomy for $\gamma=+\infty$ if there exist $\widehat{\gamma}\in\mathbb{R}$ and $\alpha<0$ such that the $\widehat{\gamma}$-weighted system admits $\muD$ with parameters $(\Id;\alpha,*)$. Analogously, we state that \eqref{701} has $\mu$-dichotomy for $\gamma=-\infty$ if there exist $\widehat{\gamma}\in\mathbb{R}$ and $\beta>0$ such that the $\widehat{\gamma}$-weighted system admits $\muD$ with parameters $(0;*,\beta)$. 
}
\end{remark}

Employing these concepts we can define the dichotomy spectrum.

\begin{definition}\label{dichspectrum}
    The \textbf{$\mu$-dichotomy spectrum} of \eqref{700} is the set
    \[\Sigma_\muD(A):=\{\gamma\in \mathbb{R}\cup\{\pm\infty\}: \eqref{701}\text{ does not admit }\muD\}.\]
   Moreover, its complement is the set $\rho_\muD(A)=\left(\mathbb{R}\cup\{\pm\infty\}\right)\setminus \Sigma_\muD(A)$ called the \textbf{$\mu$-resolvent set}.
\end{definition}

Note that we allow the $\pm\infty$-weighted equations to not have dichotomy. By the previously described convention, we have that $+\infty\in \Sigma_\muD(A)$ if for every $\gamma\in \mathbb{R}$ the $\gamma$-weighted system \eqref{701} does not have dichotomy with projector $\Id$. 

Similarly,  $-\infty\in \Sigma_\muD(A)$ if for every $\gamma\in \mathbb{R}$ the $\gamma$-weighted system \eqref{701} does not have dichotomy with projector $0$. 

Once again, this concept is also employed for differential equations, just by considering shifted systems from \eqref{3}, instead of weighted systems from \eqref{700}.

In order to establish a criteria for when these spectra are bounded sets, the concept of bounded growth (see \cite[Sect. 3.1]{Siegmund2002} or \cite{Coppel}) has been crucial in literature.

\begin{definition}\label{NmuGrowth}
		The system \eqref{700} has \textbf{$\mu$-bounded growth}, or just $\mu$-\textbf{growth}, if there are constants $\widehat{K}\geqslant 1$, $a\geqslant 0$ such that
		\[
		\|\Phi(k,n)\|\leqslant \widehat{K}\left(\frac{\mu(k)}{\mu(n)}\right)^{\sgn(k-n)a},\quad\forall\,k,n\in \mathbb{Z}.
		\]
	\end{definition}

Building upon this concept, together with well-established techniques from nonautonomous spectral theory, various authors have established distinct formulations of the so-called spectral theorem, that is, results characterizing the potential structure of the dichotomy spectrum. For a comprehensive account, we refer the reader to \cite{Sacker, Siegmund2002, Silva} and the references therein. In particular, we refer to \cite[Theorem 5.12]{Kloeden} and \cite[Theorem 4.24]{Rasmussen} regarding the convention involving $\pm\infty$, whose notation we adopt throughout this work:
$$[-\infty,a]:=(-\infty,a]\cup\{-\infty\},\qquad[a,+\infty]:=[a,+\infty)\cup\{+\infty\},$$
for an arbitrary chosen $a\in\mathbb{R}$, and
$$[-\infty,-\infty]:=\{-\infty\},\qquad[+\infty,+\infty]:=\{+\infty\},\qquad[-\infty,+\infty]=\mathbb{R}\cup \{\pm \infty\}.$$

\begin{theorem}
    (Spectral Theorem) Fix a growth rate $\mu$. Consider the nonautonomous difference equation \eqref{700}. There exists some $m\in \{1,\dots,d\}$ such that
		\[
		\Si_\muD(A)=\bigcup_{i=1}^m[a_i,b_i],
		\]
		for some $a_i,b_i\in \mathbb{R}\cup \{\pm \infty\}$ verifying $a_i\leqslant b_i<a_{i+1}$. The sets $[a_i,b_i]$ are referred to as spectral intervals, while the connected components of the resolvent set are called spectral gaps.

The following are also verified:
\begin{itemize}
    \item [(i)] For each $\gamma\in \rho_{\muD}(A)$, the $\gamma$-weighted system has $\muD$ with a unique invariant projector $n\mapsto \P_\gamma(n)$. 
    \item [(ii)] For $\zeta,\gamma\in \rho_{\muD}(A)$ with $\zeta>\gamma$, the projectors of their respective weighted systems are the same if and only if they lie in the same spectral gap.  In other words, $\P_\zeta(n)=\P_\gamma(n)$ for every $n\in \mathbb{Z}$ if and only if $[\gamma,\zeta]\subset \rho_{\muD}(A)$.
    \item [(iii)] If $\zeta,\gamma\in \rho_\muD(A)$ lie on different spectral gaps and $\zeta>\gamma$, then $\ker \P_\zeta(n)\subsetneq \ker\P_\gamma(n)$ and $\im\,\P_\gamma(n)\subsetneq \im\,\P_\zeta(n)$ for every $n\in \mathbb{Z}$.
    \item [(iv)] If the system \eqref{700} has $\mu$-growth, the spectrum is a bounded set. 
\end{itemize}
\end{theorem}

\section{Growth rates comparison}\label{section3}

This section is devoted to introducing a novel notion of comparison between growth rates. We provide examples to clarify the concept and examine its implications for the spectral behavior of a nonautonomous linear system. In particular, the proposed notion enables us to establish a duality between $\mu$-dichotomy and $\mu$-bounded growth.

\begin{definition}\label{901}
    Consider two growth rates $\mu,\omega:\mathbb{Z}\to \mathbb{R}^+$. We say $\mu$ is \textbf{faster} than $\omega$ ($\mu\gg\omega$) if for every $\alpha,\widetilde{\alpha}<0$ there exists $M\geqslant1$ such that
    \begin{equation}\label{800}
        \left(\frac{\mu(k)}{\mu(n)}\right)^\alpha\leqslant M\left(\frac{\omega(k)}{\omega(n)}\right)^{\widetilde{\alpha}},\quad \forall\,k\geqslant n.
    \end{equation}

    Equivalently, in this case we say $\omega$ is \textbf{slower} than $\mu$ ($\omega\ll\mu$).
\end{definition}

First, let us show a simple lemma to state an equivalent definition to this concept.

\begin{lemma}\label{4}
 Consider two growth rates $\mu,\omega:\mathbb{Z}\to \mathbb{R}^+$. We have that $\mu$ is faster than $\omega$ if and only if for every 
 $\beta,\widetilde{\beta}>0$ there exist $M\geqslant1$ such that
 \begin{equation}\label{801}
     \left(\frac{\mu(k)}{\mu(n)}\right)^\beta\leqslant M\left(\frac{\omega(k)}{\omega(n)}\right)^{\widetilde{\beta}},\quad \forall\,n\geqslant k.
 \end{equation}
\end{lemma}

\begin{proof}
It is enough to see that $\left(\frac{\mu(k)}{\mu(n)}\right)^\beta=\left(\frac{\mu(n)}{\mu(k)}\right)^{-\beta}$, thus \eqref{800} is equivalent to \eqref{801} by setting $\beta=-\alpha$ and $\widetilde{\beta}=-\widetilde{\alpha}$.
\end{proof}

We now proceed to illustrate this new concept through a series of examples.

\begin{example}\label{quadratic}
    {\rm
    
Consider the functions $\exp(n) = e^n$ and $\qeg(n) = e^{\sgn(n) n^2}$, for all $n \in \mathbb{Z}$, denoting the exponential and quadratic exponential growth rates, respectively. Fix $\alpha,\widetilde{\alpha}<0$ and consider $k\geqslant n\geqslant 0$. We have
\begin{eqnarray*}
    \left(\frac{\qeg(k)}{\qeg(n)}\right)^\alpha\leqslant M\left(\frac{\exp(k)}{\exp(n)}\right)^{\walpha}&\Longleftrightarrow& e^{\alpha(k^2-n^2)}\leqslant Me^{\walpha(k-n)}\\
    &\Longleftrightarrow& e^{\alpha(k^2-n^2)-\walpha(k-n)}\leqslant M\\
        &\Longleftrightarrow& \alpha(k^2-n^2)-\walpha(k-n)\leqslant \log\left(M\right)\\
        &\Longleftrightarrow& (k-n)\left(\alpha(k+n)-\walpha\right)\leqslant \log\left(M\right).
\end{eqnarray*}

Now, if $k+n\geqslant\frac{\walpha}{\alpha}$, then $\alpha(k+n)-\walpha\leqslant 0$, thus
\[
(k-n)\left(\alpha(k+n)-\walpha\right)\leqslant 0=\log\left(1\right).
\]

On the order hand, note that the set $S:=\{(k,n)\in \mathbb{Z}^2: 0\leqslant n\leqslant k\land k+n< \frac{\walpha}{\alpha}\}$ is either finite or empty. If $S=\emptyset$, there is nothing to prove. On the other hand, if $S\neq \emptyset$, we can define $\widetilde{M}_{\alpha,\walpha}>0$ by
\[
\log\left(\widetilde{M}_{\alpha,\walpha}\right):=\max_{(k,n)\in S}(k-n)\left(\alpha(k+n)-\walpha\right),
\]
and $M_{\alpha,\walpha}=\max\{1,\widetilde{M}_{\alpha,\walpha}\}$, which leads us to
\[
\left(\frac{\qeg(k)}{\qeg(n)}\right)^\alpha\leqslant M_{\alpha,\walpha}\left(\frac{\exp(k)}{\exp(n)}\right)^{\walpha},\qquad\forall\,0\leqslant n\leqslant k.
\]

Now, if $n\leqslant k \leqslant 0$, then $0\leqslant -k\leqslant -n$, thus the previous case implies
\begin{align*}
    e^{\alpha((-n)^2-(-k)^2)}\leqslant M_{\alpha,\walpha}\,e^{\walpha((-n)-(-k))}&\Longrightarrow \, e^{\alpha(n^2-k^2)}\leqslant M_{\alpha,\walpha}\,e^{-\walpha(n-k)}\\
    &\Longrightarrow \, e^{\alpha(\sgn(k)k^2-\sgn(n)n^2)}\leqslant M_{\alpha,\walpha}\,e^{\walpha(k-n)}\\
    &\Longrightarrow \, \left(\frac{\qeg(k)}{\qeg(n)}\right)^\alpha\leqslant M_{\alpha,\walpha}\left(\frac{\exp(k)}{\exp(n)}\right)^{\walpha}.
\end{align*}

Finally, for $n\leqslant 0\leqslant k$, we have proved
\[
\left(\frac{\qeg(k)}{\qeg(0)}\right)^\alpha\leqslant M_{\alpha,\walpha}\left(\frac{\exp(k)}{\exp(0)}\right)^{\walpha}\quad\text{and}\quad\left(\frac{\qeg(0)}{\qeg(n)}\right)^\alpha\leqslant M_{\alpha,\walpha}\left(\frac{\exp(0)}{\exp(n)}\right)^{\walpha},
\]
therefore, multiplying these expressions we obtain
\[
\left(\frac{\qeg(k)}{\qeg(n)}\right)^\alpha\leqslant M_{\alpha,\walpha}^2\left(\frac{\exp(k)}{\exp(n)}\right)^{\walpha},
\]
hence, we conclude
\[
\left(\frac{\qeg(k)}{\qeg(n)}\right)^\alpha\leqslant M_{\alpha,\walpha}^2\left(\frac{\exp(k)}{\exp(n)}\right)^{\walpha},\qquad\forall\, n\leqslant k.
\]
    }
\end{example}





\begin{example}\label{cubic}
    {\rm 
    Consider the quadratic and cubic growth rates defined respectively by $\qeg(n)=e^{\sgn(n)n^2}$ and $\ceg(n)=e^{n^3}$, for all $n\in\mathbb{Z}$. We claim that $\ceg$ is faster than $\qeg$. Indeed, let $\alpha,\widetilde{\alpha}<0$ be fixed and assume that $k\geqslant n\geqslant0$. Observe that the following equivalence holds: 
    \begin{eqnarray*}
    \left(\frac{\ceg(k)}{\ceg(n)}\right)^\alpha\leqslant M\left(\frac{\qeg(k)}{\qeg(n)}\right)^{\walpha}&\Longleftrightarrow& 
         \alpha(k^3-n^3)-\walpha(k^2-n^2)\leqslant \log\left(M\right)\\
        &\Longleftrightarrow& (k-n)\left(\alpha(k^2+kn+k^2)-\walpha(k+n)\right)\leqslant \log\left(M\right).
\end{eqnarray*}
Therefore, to ensure that $\qeg \ll \ceg$, it is necessary to define an appropriate constant $M \geqslant 1$ that satisfies the above inequality for all $k \geqslant n \geqslant 0$.

Note that $k + n = 0$ if and only if $k = n = 0$, in which case the inequality holds trivially by choosing $M = 1$. Now, if $(k+n)-\dfrac{kn}{k+n}\geqslant\dfrac{\walpha}{\alpha}$, then it follows that
\[
(k-n)\left(\alpha(k^2+kn+k^2)-\walpha(k+n)\right)\leqslant0.
\]
Therefore, once again, it suffices to take $M = 1$.

On the other hand, similarly as in Ex.~\ref{quadratic}, consider the set
\[
\mathcal{S}=\left\{(k,n)\in\mathbb{Z}^{2}: 0\leqslant n \leqslant k \, \wedge\ \, (k+n)-\dfrac{kn}{k+n}<\dfrac{\walpha}{\alpha}\right\}.
\]
This set is either finite or empty. In the case where $\mathcal{S}$ is empty, there is nothing to prove, as we fall under the previous cases. If $\mathcal{S}$ is finite, we define
\[
M_{\alpha,\walpha}:=\exp\left(\max_{(k,n)\in\mathcal{S}}(k-n)\left(\alpha(k^2+kn+k^2)-\walpha(k+n)\right)\right),
\]
and set $M := \max\{1, {M}_{\alpha,\walpha}\}$. This ensures that the first inequality holds, thereby establishing the desired result. 

The scenarios $k \geqslant 0 \geqslant n$ and $0 \geqslant k \geqslant n$ proceed similarly following the reasoning established in Ex.~\ref{quadratic}.
    }
\end{example}

\begin{example}\label{poli}
{\rm
Consider $\exp$ and $p$, the exponential and polynomial growth rates, respectively. We have $p\ll\exp$. Indeed, fix $\alpha,\walpha<0$ and consider $k\geqslant n\geqslant 0$. Let us consider five cases:
\begin{itemize}
    \item Case 1: $n=0$. Then, we define
    $$M_1=\max\left\{1,\,\sup_{k\geqslant 1}\left\{k^{-\walpha}\cdot e^{\alpha k}\right\}\right\}<+\infty,$$
    and we obtain
        \begin{align*}
        e^{\alpha(k-0)}\leqslant M_1  \left(\frac{p(k)}{p(0)}\right)^\walpha.
    \end{align*}

    \item Case 2: $k=n$ and $n\geqslant 1$. Then
    \begin{align*}
    1\leqslant M_1\Rightarrow    e^{\alpha(n-n)}\leqslant M_1 \left(\frac{n}{n}\right)^\walpha=M_1 \left(\frac{p(n)}{p(n)}\right)^\walpha.
    \end{align*}

    \item Case 3: $k=n+1$ and $n\geqslant 1$. Define
    \[
    M_2=\sup_{n\geqslant 1}\left\{\left(\frac{n+1}{n}\right)^{-\walpha}\cdot e^\alpha\right\}=\left(\frac{3}{2}\right)^{-\walpha}e^\alpha<+\infty,
    \]
and we have
\begin{align*}
        e^{\alpha(n+1-n)}=e^\alpha\leqslant M_2 \left(\frac{n+1}{n}\right)^\walpha=M_2 \left(\frac{p(n+1)}{p(n)}\right)^\walpha.
    \end{align*}

\item Case 4: $k=3$ and $n=1$. Define $M_3=3^{-\walpha}\cdot e^{2\alpha}$. We obtain 
\begin{align*}
        e^{\alpha(3-1)}=e^\alpha\leqslant M_3 \left(\frac{3}{1}\right)^\walpha=M_3 \left(\frac{p(3)}{p(1)}\right)^\walpha.
    \end{align*}

    \item Case 5: $k\geqslant n+2$ and $n\geqslant 2$. Then we have $n+(k-n)\leqslant n(k-n)$, thus 
\begin{align*}
    \frac{n+(k-n)}{n}\leqslant k-n \, &\Longrightarrow  \left(\frac{n+(k-n)}{n}\right)^{-\walpha}\leqslant (k-n)^{-\walpha}\\
    &\Longrightarrow e^{\alpha(k-n)}\left(\frac{n+(k-n)}{n}\right)^{-\walpha}\leqslant (k-n)^{-\walpha}e^{\alpha(k-n)}\leqslant M_1\\
    &\Longrightarrow e^{\alpha(k-n)}\leqslant M_1 \left(\frac{k}{n}\right)^\walpha=M_1\left(\frac{p(k)}{p(n)}\right)^\walpha.
\end{align*}
\end{itemize}

Then, it is enough to define $M=\max\{M_1,M_2,M_3\}$, and we have obtained
\[
e^{\alpha(k-n)}\leqslant M\left(\frac{p(k)}{p(n)}\right)^\walpha,\qquad\forall\,0\leqslant n\leqslant k.
\]
The other cases are handled similarly to Example \ref{quadratic}, thus proving the claim.
}
\end{example}

\begin{figure}[h]
		\centering
		\includegraphics[width=0.60\textwidth]{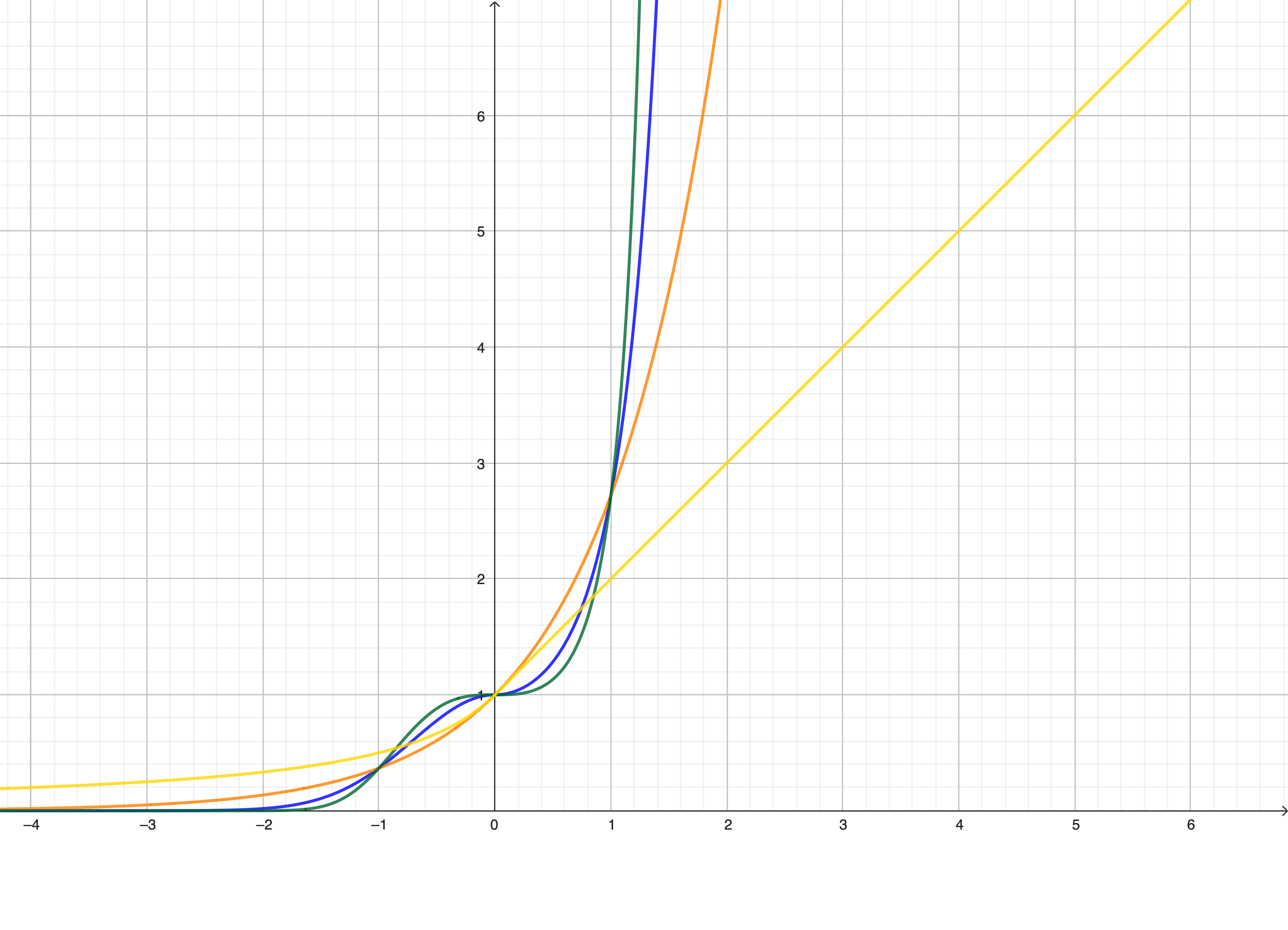}
		\caption{graphs of the exponential (orange), quadratic exponential (blue), cubic exponential (green) and polynomial (yellow) growth rates.}
\end{figure}

\begin{remark}\label{807}
{\rm
Note that $\ll$ defines a transitive relation in the collection of growth rates. In particular, as Example \ref{quadratic} proved $\exp\ll\qeg$, Example \ref{cubic} proved $\qeg\ll\ceg$ and  Example \ref{poli} proved $p\ll\exp$, then we can write $p\ll\exp\ll\qeg\ll\ceg$.
}
\end{remark}

In the context of differential equations, analogous proofs to the given above demonstrate again $p\ll\exp\ll\qeg\ll\ceg$. The main difference is that in the continuous time framework the finite sets employed in the proofs transform into compact sets.

In the following, we study the consequences of the existence of dichotomy for a slower growth rate.

\begin{theorem}\label{805}
    Consider two growth rates $\mu,\omega$ such that $\mu$ is faster than $\omega$. Suppose system \eqref{700} has $\muD$. Then, its $\omega\mathrm{D}$ spectrum is either $\{+\infty\}$, $\{-\infty\}$ or $\{\pm\infty\}$.
\end{theorem}

\begin{proof}
  Suppose \eqref{700} has $\muD$ with parameters $(\P;\alpha,\beta)$ and constant $K\geqslant1$.  Let $\gamma\in \mathbb{R}$ and choose $\walpha<\min\{\gamma,0\}$ and $\wbeta>\max\{\gamma,0\}$. Suppose first that $0\neq\P\neq \Id$. 
  
  As $\mu$ is faster than $\omega$, there is $M\geqslant 1$ such that
 \[\left(\frac{\mu(k)}{\mu(n)}\right)^\alpha\leqslant M\left(\frac{\omega(k)}{\omega(n)}\right)^{\widetilde{\alpha}},\quad \forall\,k\geqslant n,\]
 and 
    \[\left(\frac{\mu(k)}{\mu(n)}\right)^\beta\leqslant M\left(\frac{\omega(k)}{\omega(n)}\right)^{\widetilde{\beta}},\quad \forall\,n\geqslant k.\]

    Then, we have
    \begin{eqnarray*}
        \|\Phi_{\omega,\gamma}(k,n)\P(n)\|=\|\Phi(k,n)\P(n)\|\cdot\left(\frac{\omega(k)}{\omega(n)}\right)^{-\gamma}\leqslant  KM\left(\frac{\omega(k)}{\omega(n)}\right)^{\widetilde{\alpha}-\gamma},\quad \forall\,k\geqslant n,
    \end{eqnarray*}
    and
    \begin{align*}
        \|\Phi_{\omega,\gamma}(k,n)[\Id-\P(n)]\|&=\|\Phi(k,n)[\Id-\P(n)]\|\cdot\left(\frac{\omega(k)}{\omega(n)}\right)^{-\gamma}\\
        &\leqslant  KM\left(\frac{\omega(k)}{\omega(n)}\right)^{\wbeta-\gamma},\quad \forall\,n\geqslant k,
    \end{align*}
    which means that the $(\omega,\gamma)$-weighted system admits $\omega$-dichotomy, or equivalently $\gamma\in \rho_{\omega\mathrm{D}}(A)$. Since $\gamma\in \mathbb{R}$ was arbitrary, $0\neq\P\neq \Id$ and the projectors associated to $\mu$-dichotomies are unique, then $\Sigma_{\omega\mathrm{D}}(A)=\{\pm\infty\}$.

    Finally, if $\P=\Id$, an analogous argument shows that $\Sigma_\omegaD(A)=\{-\infty\}$ and if $\P=0$ then $\Sigma_\omegaD(A)=\{+\infty\}$.
\end{proof}



\begin{corollary}\label{721}
    If system \eqref{700} has $\muD$, then it does not have $\omega$-bounded growth for any growth rate $\omega$ which is slower than $\mu$.
\end{corollary}
\begin{proof}
    This follows from the fact that if \eqref{700} had $\omega$-growth for some $\omega\ll\mu$, then it has bounded $\omega\mathrm{D}$ spectrum.
\end{proof}

\begin{example}
 {\rm   Consider the scalar equation $x(n+1)=\mathfrak{a}(n)x(n)$, where
    \[\mathfrak{a}(n)=\exp\left(-3n^2-3n-1\right),\]
    with evolution operator $\Phi(k,n)=e^{-(k^3-n^3)}$. If we consider the cubic exponential growth rate $\ceg(n)=e^{n^3}$, it clearly has $\cD$ with parameters $(\Id;1,*)$. 
    
    As established in Remark \ref{807}, the exponential growth rate ($\exp$) satisfies $\exp \ll \ceg$. Hence, Theorem \ref{805} implies that $\Sigma_{\exp\mathrm{D}}(\mathfrak{a}) \subset \{\pm\infty\}$. Moreover, we can check that for every $\gamma\in \mathbb{R}$, the invariant projector associated to the $(\exp,\gamma)$-weighted system is the identity. Thus, we have $\Sigma_\eD(\mathfrak{a})=\{-\infty\}$.

    Moreover, since the system satisfies $\ceg$-growth, Corollary \ref{721} allows us to conclude that it does not satisfy $\exp$-growth.

    }
\end{example}

    

Now we delve into a sort of dual problem, where we consider how the fact that a system has bounded growth of a certain growth rate affects all other faster growth rates.

\begin{theorem}\label{806}
    If \eqref{700} has $\omega$-bounded growth and $\mu$ is faster than $\omega$, then $\Sigma_\muD(A)=\{0\}$. 
\end{theorem}

\begin{proof}
If \eqref{700} has $\omega$-bounded growth, there is some $a>0$ and $K\geqslant 1$ such that
\[\|\Phi(k,n)\|\leqslant \left(\frac{\omega(k)}{\omega(n)}\right)^{\sgn(k-n)a},\quad\forall\,k,n\in \mathbb{Z}.
\]
In particular,
\[\|\Phi(k,n)\|\leqslant \left(\frac{\omega(k)}{\omega(n)}\right)^{a},\quad\forall\,k\geqslant n.\]
Hence, for $\gamma>0$
\[\|\Phi_{\mu,\gamma}(k,n)\|\leqslant \left(\frac{\omega(k)}{\omega(n)}\right)^{a}\left(\frac{\mu(k)}{\mu(n)}\right)^{-\gamma},\quad\forall\,k\geqslant n.
\]

Now, if we choose $\alpha\in(-\gamma,0)$, then there exists $M\geqslant 1$ such that
        \[\left(\frac{\mu(k)}{\mu(n)}\right)^{\alpha}\leqslant M\left(\frac{\omega(k)}{\omega(n)}\right)^{-a},\quad \forall\,k\geqslant n,\]
or equivalently
        \[ \left(\frac{\omega(k)}{\omega(n)}\right)^{a}  \leqslant M \left(\frac{\mu(k)}{\mu(n)}\right)^{-\alpha},\quad \forall\,k\geqslant n.\]
        From which we derive
        \[
        \|\Phi_{\mu,\gamma}(k,n)\|\leqslant M\left(\frac{\mu(k)}{\mu(n)}\right)^{-\gamma-\alpha},\quad\forall\,k\geqslant n.
        \]
       Consequently, the $(\mu,\gamma)$-weighted system admits $\muD$ with parameters $(\Id; -\gamma - \alpha, *)$, implying that $\gamma \in \rho_{\muD}(A)$ and $(0, +\infty) \subset \rho_{\muD}(A)$, with $\Id$ as the corresponding invariant projector. An analogous argument shows that $(-\infty, 0) \subset \rho_{\muD}(A)$, with the zero operator as the invariant projector associated with this spectral gap. Therefore, as there can not be only one spectral gap, as they have different projector, we have $\Sigma_\muD(A)=\{0\}$.
\end{proof}

As a direct consequence, we obtain the following corollary.

\begin{corollary}\label{722}
    If system \eqref{700} has $\omega$-bounded growth, then it does not have $\mu$-dichotomy for any growth rate $\mu$ which is faster than $\omega$.
\end{corollary}

\begin{remark}\label{907}
{\rm
Corollaries \ref{721} and \ref{722} highlight a duality between the concepts of bounded growth and dichotomy. From a qualitative perspective, having $\mu$-growth is an indicator of how slow the growth of the system is, while having $\mu$-dichotomy is an indicator of how fast the decay on the system is. 

In order to manage both concepts, some authors have defined the notion of {\it strong} dichotomy (see 
\cite[Def 2.2]{BV} for a definition on the nonuniform framework). It is proved that a system exhibits strong dichotomy if and only if it verifies both dichotomy and bounded growth.
    }
\end{remark}

Let us illustrate this duality with examples:

\begin{example}
    {\rm
    Consider the following non autonomous differential systems
\begin{itemize}
    \item [i)] $\dot{x}=\frac{1}{1+|t|}x$.
    
    \smallskip
    
\begin{itemize}
    \item [$\bullet$] Has cubic, quadratic, exponential and polynomial growth.
    \item [$\bullet$] Has polynomial dichotomy, but not exponential, quadratic nor cubic exponential dichotomy.
    \item [$\bullet$] Has strong polynomial dichotomy.
    \item [$\bullet$] $\Sigma_{p\mathrm{D}}\left(\frac{1}{1+|t|}\right)=\{1\}\quad\text{and}\quad\Sigma_{\ED}\left(\frac{1}{1+|t|}\right)=\Sigma_{\qeg\mathrm{D}}\left(\frac{1}{1+|t|}\right)=\Sigma_{\ceg\mathrm{D}}\left(\frac{1}{1+|t|}\right)=\{0\}$.
\end{itemize}

\smallskip

\item  [ii)] $\dot{x}=2|t|x$.

\smallskip

\begin{itemize}
    \item [$\bullet$] Has and cubic and quadratic growth, but not polynomial nor exponential growth.
    \item [$\bullet$] Has polynomial, exponential and quadratic dichotomy, but not cubic dichotomy.
    \item [$\bullet$] Has strong quadratic dichotomy.
    \item [$\bullet$] $\Sigma_{p\mathrm{D}}\left(2|t|\right)\!=\!\Sigma_{\ED}\left(2|t|\right)\!=\!\{+\infty\}$,\,\,\, $\Sigma_{\qeg\mathrm{D}}\left(2|t|\right)\!=\!\{1\}$ \,\,\,and  \,\,\,$\Sigma_{\ceg\mathrm{D}}\left(2|t|\right)\!=\!\{0\}$.
\end{itemize}

\smallskip

\item [iii)] $\dot{x}=3t^2x$.

\smallskip

\begin{itemize}
    \item [$\bullet$] Has cubic growth, but not polynomial, exponential nor quadratic growth.
    \item [$\bullet$] Has cubic, quadratic, exponential and polynomial dichotomy.
    \item [$\bullet$] Has strong cubic dichotomy.
    \item [$\bullet$] $\Sigma_{p\mathrm{D}}(3t^2)=\Sigma_{\ED}(3t^2)=\Sigma_{\qeg\mathrm{D}}(3t^2)=\{+\infty\}\quad\text{and}\quad \Sigma_{\ceg\mathrm{D}}(3t^2)=\{1\}$.
\end{itemize}
\end{itemize}
    }
\end{example}


\section{A weak comparison}\label{section4}

In this section, we present and investigate a weak notion of comparison between growth rates, and derive its spectral implications.

\begin{definition}\label{902}
    Consider two growth rates $\mu,\omega:\mathbb{Z}\to \mathbb{R}^+$. We say $\mu$ is \textbf{weakly faster} than $\omega$ ($\mu\succ\omega$) if for every $\alpha<0$ there exist $M\geqslant1$ such that
    \begin{equation}\label{802}
        \left(\frac{\mu(k)}{\mu(n)}\right)^\alpha\leqslant M\left(\frac{\omega(k)}{\omega(n)}\right)^{\alpha},\quad \forall\,k\geqslant n.
    \end{equation}

    Equivalently, in this case we say $\omega$ is \textbf{weakly slower} than $\mu$ ($\omega\prec\mu$).
\end{definition}

An immediate remark is that $\omega\ll\mu$ implies $\omega\prec\mu$. Nevertheless, the converse is not true.

\begin{example}\label{810}
    {\rm Let $\omega:\mathbb{Z}\to\mathbb{R}$ be a growth rate and consider $\mu(k)=(\omega(k))^{\theta}$, with $\theta\geqslant1$, for all $k\in\mathbb{Z}$. One can readily verify that $\omega\prec\mu$. Indeed, let $\alpha<0$ be fixed and consider $k\geqslant n$. Since
     \begin{eqnarray*}
    \left(\frac{\mu(k)}{\mu(n)}\right)^\alpha\leqslant M\left(\frac{\omega(k)}{\omega(n)}\right)^{\alpha}&\Longleftrightarrow& 
         \alpha(\theta-1)\log\left(\frac{\omega(k)}{\omega(n)}\right)\leqslant \log(M),
\end{eqnarray*}
it is enough to consider $M=1$ and the inequality \eqref{802} trivially holds. Moreover, the reader can verify that $\omega\not\ll\mu$.}
\end{example}


An analogue of Lemma~\ref{4} is derived concerning weakly faster growth rates.

\begin{remark}\label{905}
    {\rm $\mu$ is weakly faster than $\omega$ if and only  if there exist $0<m<1$ such that
    \begin{equation*}
     m\cdot \frac{\omega(t)}{\omega(s)} \leq  \frac{\mu(t)}{\mu(s)},\quad \forall\,t\geq s.
    \end{equation*}
or equivalently
\begin{equation*}
     m\cdot \frac{\mu(s)}{\omega(s)} \leq  \frac{\mu(t)}{\omega(s)},\quad \forall\,t\geq s.
    \end{equation*}

In particular, if $t\mapsto \frac{\mu(t)}{\omega(t)}$ is increasing, then $\omega\prec\mu$.

 
 }
\end{remark}

We now focus on the main objective of this section, namely the analysis of the spectral consequences that follow from this weak comparison.

\begin{theorem}\label{808}
    Consider two growth rates $\mu,\omega:\mathbb{Z}\to\mathbb{R}^+$ such that $\mu$ is weakly faster than $\omega$.  Fix $a>0$. Then, for system \eqref{700}:
    \begin{itemize}
        \item [i)] If $\Sigma_\muD(A)\subset [-\infty,-a]$, then $\Sigma_\omegaD(A)\subset [-\infty,-a]$,
        \item [ii)] If $\Sigma_\muD(A)\subset [a,+\infty]$, then $\Sigma_\omegaD(A)\subset [a,+\infty]$.
    \end{itemize}
\end{theorem}

\begin{proof}
We will only prove i) since ii) is analogous. If $\gamma\in(-a,0)$, then the $(\mu,\gamma)$-weighted has $\muD$ with parameters $(\Id;\alpha,*)$, for some $\alpha<0$. As $\omega\prec\mu$, then there exists $M\geqslant 1$ such that
 \begin{equation*}
        \left(\frac{\mu(k)}{\mu(n)}\right)^\alpha\leqslant M\left(\frac{\omega(k)}{\omega(n)}\right)^{\alpha}\quad\text{and}\quad\left(\frac{\mu(k)}{\mu(n)}\right)^\gamma\leqslant M\left(\frac{\omega(k)}{\omega(n)}\right)^{\gamma},\qquad \forall\,k\geqslant n.
    \end{equation*}
Then we get
      \begin{align*}
        \|\Phi_{\omega,\gamma}(k,n)\|&=\|\Phi(k,n)\|\cdot\left(\frac{\omega(k)}{\omega(n)}\right)^{-\gamma}\\
        &=\|\Phi_{\mu,\gamma}(k,n)\|\cdot  \left(\frac{\mu(k)}{\mu(n)}\right)^{\gamma}\left(\frac{\omega(k)}{\omega(n)}\right)^{-\gamma}\\
        &\leqslant K\cdot \left(\frac{\mu(k)}{\mu(n)}\right)^{\alpha} \left(\frac{\mu(k)}{\mu(n)}\right)^{\gamma}\left(\frac{\omega(k)}{\omega(n)}\right)^{-\gamma}\\
        &\leqslant M^2 K\cdot \left(\frac{\omega(k)}{\omega(n)}\right)^{\alpha} \left(\frac{\omega(k)}{\omega(n)}\right)^{\gamma}\left(\frac{\omega(k)}{\omega(n)}\right)^{-\gamma}\\
        &= M^2 K\cdot \left(\frac{\omega(k)}{\omega(n)}\right)^{\alpha},\quad \forall\,k\geqslant n.
    \end{align*}
Hence, we infer that $\gamma \in \rho_{\muD}(A)$, with $\Id$ as the corresponding invariant projector. This implies that $(-a, 0) \subset \rho_{\muD}(A)$ with invariant projector $\Id$. Consequently, $(-a, +\infty] \subset \rho_{\muD}(A)$, which leads to the inclusion $\Sigma_{\muD}(A) \subset [-\infty, -a]$.
\end{proof}

In particular, the previous result shows that if \eqref{700} admits $\muD$, then it immediately admits $\omegaD$. 

The following two figures illustrate the behavior described in items (i) and (ii) of Theorem~\ref{808}, respectively.

\begin{center}
    \begin{tikzpicture}[scale=1.0, thick]
  \draw[<->] (-4,0) -- (5,0);
  \node at (-4,0.5) {$-\infty$};
  
  \draw[blue, thick] (-1.5,0.75) -- (-1.5,-0.75);
  \draw[blue, thick] (-1.5,0.75) -- (-1.6,0.75);
  \draw[blue, thick] (-1.5,-0.75) -- (-1.6,-0.75);
  \node[blue] at (-1.0,1.0) {$\Sigma_{\omegaD}(A)$};

  \draw[red, thick] (0.3,0.75) -- (0.3,-0.75);
  \draw[red, thick] (0.3,0.75) -- (0.2,0.75);
  \draw[red, thick] (0.3,-0.75) -- (0.2,-0.75);
  \node[red] at (-0.1,-1.0) {$\Sigma_{\muD}(A)$};

  \draw[blue, ->, thick] (-2.0,0.2) -- (-2.6,0.2);
  \draw[blue, ->, thick] (-2.0,-0.2) -- (-2.6,-0.2);

  \draw[red, ->, thick] (-0.3,0.2) -- (-0.9,0.2);
  \draw[red, ->, thick] (-0.3,-0.2) -- (-0.9,-0.2);

  \foreach \x/\label in {2.5/{$-a$}, 3.5/{$0$}, 4.5/{$a$}} {
    \draw[black] (\x,0.2) -- (\x,-0.2);
    \node at (\x,-0.5) {\label};
  }

\end{tikzpicture}
\end{center}

\begin{center}
    \begin{tikzpicture}[scale=1.0, thick]
  \draw[<->] (3,0) -- (12,0);
  \node at (12,0.5) {$+\infty$};

  \draw[blue, thick] (8.5,0.75) -- (8.5,-0.75);
  \draw[blue, thick] (8.5,0.75) -- (8.6,0.75);
  \draw[blue, thick] (8.5,-0.75) -- (8.6,-0.75);
  \node[blue] at (9.0,1.0) {$\Sigma_{\omegaD}(A)$}; 

  \draw[red, thick] (6.7,0.75) -- (6.7,-0.75);
  \draw[red, thick] (6.7,0.75) -- (6.8,0.75);
  \draw[red, thick] (6.7,-0.75) -- (6.8,-0.75);
  \node[red] at (7.1,-1.0) {$\Sigma_{\muD}(A)$}; 

  \foreach \x/\label in {3.5/{$-a$}, 4.5/{$0$}, 5.5/{$a$}} {
    \draw[black] (\x,0.2) -- (\x,-0.2);
    \node at (\x,-0.5) {\label};
  }

  \draw[blue, ->, thick] (9.0,0.2) -- (9.6,0.2);
  \draw[blue, ->, thick] (9.0,-0.2) -- (9.6,-0.2);

  \draw[red, ->, thick] (6.8,0.2) -- (7.4,0.2);
  \draw[red, ->, thick] (6.8,-0.2) -- (7.4,-0.2);
\end{tikzpicture}
\end{center}

Later, we will see how these results relate to the conclusions of the previous section. For this goal, consider the following notations:

$$\Sigma_\muD^+(A)=\Sigma_\muD(A)\cap [0,+\infty],\qquad\Sigma_\muD^-(A)=\Sigma_\muD(A)\cap [-\infty,0]. $$

\begin{theorem}\label{809}
Consider two growth rates $\mu,\omega:\mathbb{Z}\to\mathbb{R}^+$ such that $\mu$ is weakly faster than $\omega$.  Fix $a$, $b$ with $a\leqslant 0\leqslant b$. Then, for system \eqref{700}
\begin{itemize}
    \item [i)] if $\Sigma_\omegaD^+(A)\subset[0,b]$ then $\Sigma_\muD^+(A)\subset[0,b]$,
    \item [ii)] if $\Sigma_\omegaD^-(A)\subset[a,0]$ then $\Sigma_\muD^-(A)\subset[a,0]$,
    \item [iii)] if $\Sigma_\omegaD(A)\subset[a,b]$ then $\Sigma_\muD(A)\subset[a,b]$.
\end{itemize}
\end{theorem}

\begin{proof}
It suffices to prove item i), as the remaining cases are analogous. If $b= +\infty$, there is nothing to prove. If $b<+\infty$, then $+\infty\not\in \Sigma_\omegaD(A)$. This implies that for some $\zeta\in \mathbb{R}$, the $(\omega,\zeta)$-weighted system has $\omegaD$ with parameters $(\Id;\alpha,*)$ and constant $K\geqslant1$ for some $\alpha=\alpha(\zeta)<0$.

Moreover, since $(b, +\infty] \subset \rho_{\omegaD}(A)$, for every $\zeta > b$, the $(\omega, \zeta)$-weighted system admits $\omegaD$ with parameters $(\Id; \alpha, *)$, where $\alpha = \alpha(\zeta) < 0$. Therefore, for $\gamma>b\geqslant 0$, choose $\zeta\in (b,\gamma)$. The $(\omega,\zeta)$-weighted system has $\omegaD$ with parameters $(\Id;\alpha,*)$.

On the other hand, since $\omega\prec \mu$, there exists $M\geqslant 1$ such that
\begin{equation*}
        \left(\frac{\mu(k)}{\mu(n)}\right)^{-\zeta}\leqslant M\left(\frac{\omega(k)}{\omega(n)}\right)^{-\zeta},\quad \forall\,k\geqslant n.
    \end{equation*}

We now obtain the following estimates:
      \begin{align*}
        \|\Phi_{\mu,\zeta}(k,n)\|&=\|\Phi(k,n)\|\cdot\left(\frac{\mu(k)}{\mu(n)}\right)^{-\zeta}\\
        &\leqslant\|\Phi_{\omega,\zeta}(k,n)\|\cdot  \left(\frac{\omega(k)}{\omega(n)}\right)^{\zeta}\left(\frac{\mu(k)}{\mu(n)}\right)^{-\zeta}\\
        &\leqslant K\cdot \left(\frac{\omega(k)}{\omega(n)}\right)^{\alpha} \left(\frac{\omega(k)}{\omega(n)}\right)^{\zeta}\left(\frac{\mu(k)}{\mu(n)}\right)^{-\zeta}\\
        &\leqslant KM\cdot \left(\frac{\omega(k)}{\omega(n)}\right)^{\alpha} \left(\frac{\omega(k)}{\omega(n)}\right)^{\zeta}\left(\frac{\omega(k)}{\omega(n)}\right)^{-\zeta}\\
        &\leqslant KM,\qquad \forall\,k\geqslant n.
    \end{align*}
Finally, using the identity  $\Phi_{\mu,\gamma}(k,n)=\Phi_{\mu,\zeta}(k,n)\left(\frac{\mu(k)}{\mu(n)}\right)^{\zeta-\gamma}$ , we conclude that
\[
\|\Phi_{\mu,\gamma}(k,n)\|\leqslant KM \left(\frac{\mu(k)}{\mu(n)}\right)^{\zeta-\gamma},\quad \forall\,k\geqslant n.
\]
Hence, the $(\mu,\gamma)$-weighted system has $\muD$ with parameters $(\Id;\zeta-\gamma,*)$. Thus, we infer that $(b,+\infty]\subset \rho_\muD(A)$, {\it i.e.}, $\Sigma_\muD^+(A)\subset [0,b]$.
\end{proof}

The next figure illustrate the statement of Theorem~\ref{809}--(iii): 

\medskip

\begin{center}
\begin{tikzpicture}[scale=1.0, thick]
  \draw[->] (-4,0) -- (9,0); 
  \draw[<-] (-4,0) -- (5,0); 

  \draw[black] (-3.5,0.2) -- (-3.5,-0.2) node[below] {$a$};
  \draw[black] (3,0.2) -- (3,-0.2) node[below] {$0$};
  \draw[black] (8,0.2) -- (8,-0.2) node[below] {$b$};

  \draw[blue, thick] (-1.5,0.75) -- (-1.5,-0.75) -- (-1.4,-0.75);
  \draw[blue, thick] (-1.5,0.75) -- (-1.4,0.75);
  \node[blue] at (-0.5,1.0) {$\Sigma_{\omegaD}^{-}(A)$};

  \draw[->, blue, thick] (-0.8,-0.3) -- (-0.2,-0.3);

  \draw[blue, thick] (7.0,0.75) -- (7.0,-0.75) -- (6.9,-0.75);
  \draw[blue, thick] (7.0,0.75) -- (6.9,0.75);
  \node[blue] at (6.0,1.0) {$\Sigma_{\omegaD}^{+}(A)$};

  \draw[->, blue, thick] (6.3,-0.3) -- (5.7,-0.3);

  \draw[red, thick] (1.5,0.75) -- (1.5,-0.75) -- (1.6,-0.75);
  \draw[red, thick] (1.5,0.75) -- (1.6,0.75);
  \node[red] at (1.9,-1.1) {$\Sigma_{\muD}^{-}(A)$};

  \draw[red, thick] (4.5,0.75) -- (4.5,-0.75) -- (4.4,-0.75);
  \draw[red, thick] (4.5,0.75) -- (4.4,0.75);
  \node[red] at (4.5,-1.1) {$\Sigma_{\muD}^{+}(A)$};

  \draw[->, red, thick] (1.8,-0.3) -- (2.4,-0.3);
  \draw[->, red, thick] (4.2,-0.3) -- (3.5,-0.3);

\end{tikzpicture}
\end{center}

\begin{remark}
{\rm
By gathering the conclusions of Theorems \ref{805}, \ref{806}, \ref{808} and \ref{809} a clear scenario is revealed: As growth rates become faster (or weakly faster) the spectrum shrinks towards zero. On the other hand, if the growth rates become slower (or weakly slower), the spectrum expands both towards $-\infty$ and $+\infty$.  
}    
\end{remark}

\bigskip

\section{Order relations on growth rates}\label{section5}

Consider either continuous or discrete time systems and denote by $\mathcal{G}$ the family of all growth rates in that time framework. So far, we have introduced two comparison criteria, namely $\prec$ and $\ll$. In this section we look to define order relations on $\mathcal{G}$ based on these criteria and study their spectral consequences.

As mentioned on Remark \ref{807}, the relation $\ll$ is transitive. However, it its clear that it is not an order relation, since its not reflexive, as a growth rate is not faster than itself, and it is also not anti-symmetric. On the other hand, $\prec$ is clearly transitive and it is reflexive, although again not anti-symmetric.

For transitive reflexive but not anti-symmetric relations, it is standard to define an equivalence relation in order to obtain an order, and this will be the procedure we follow with the weak comparison $\prec$. On the other hand, as $\ll$ is not reflexive, we must first go through the following auxiliary concept:

\begin{definition}\label{almostfaster}
    Consider two growth rates $\mu,\omega:\mathbb{Z},\mathbb{R}\to \mathbb{R}^+$. We say $\mu$ is \textbf{almost faster} than $\omega$ ($\mu\succdot\omega$) if for every $\walpha<0$ there exist $M(\walpha)\geqslant1$ and $\alpha(\walpha)<0$ such that
    \begin{equation}\label{900}
        \left(\frac{\mu(k)}{\mu(n)}\right)^\alpha\leqslant M\left(\frac{\omega(k)}{\omega(n)}\right)^{\walpha},\quad \forall\,k\geqslant n.
    \end{equation}

    On the other hand, we say $\omega$ is \textbf{almost slower} than $\mu$ ($\omega\precdot\mu$) if for every $\alpha<0$ there exist $M(\alpha)\geqslant1$ and $\walpha(\alpha)<0$ such that \eqref{900} is verified.
\end{definition}

As established in Lemma \ref{4} and Remark \ref{905}, this relation may equivalently be formulated in terms of positive exponents.

From Definitions \ref{901} and \ref{902}, we (tautologically) have 
\[
\omega\ll\mu\Longleftrightarrow \mu\gg\omega\quad\text{and}\quad\omega\prec\mu\Longleftrightarrow \mu\succ\omega.
\]

However, this equivalence does not hold for the new Def.~\ref{almostfaster}, as despite the equation \eqref{900} remains the same, the order in which the parameters are selected differs in each case.

\begin{lemma}\label{903}
    Consider two growth rates $\mu_1,\mu_2\in\mathcal{G}$. We have:
    \begin{itemize}
         \item [i)] If $\omega\ll\mu_1$ and $\mu_1\precdot \mu_2$, then $\omega\ll\mu_2$.
    \item [ii)] If $\omega\gg\mu_1$ and $\mu_1\succdot\mu_2$, then $\omega\gg\mu_2$.
   
    \end{itemize}
\end{lemma}

\begin{proof}
    \begin{itemize}
        \item [i)] Choose $\alpha,\walpha<0$ and assume $\mu_1\precdot\mu_2$. Then, there exist $M=M(\alpha)\geqslant 1$ and $\hat{\alpha}=\hat{\alpha}(\alpha)<0$ such that
        \begin{equation*}
                    \left(\frac{\mu_2(k)}{\mu_2(n)}\right)^\alpha\leqslant M\left(\frac{\mu_1(k)}{\mu_1(n)}\right)^{\hat{\alpha}},\quad \forall\,k\geqslant n.
        \end{equation*}
Moreover, since $\omega\ll\mu_1$, there exists $\hat{M}=\hat{M}(\hat{\alpha},\walpha)\geqslant 1$ such that
        \begin{equation*}
                    \left(\frac{\mu_1(k)}{\mu_1(n)}\right)^{\hat{\alpha}}\leqslant \hat{M}\left(\frac{\omega(k)}{\omega(n)}\right)^{\walpha},\quad \forall\,k\geqslant n.
        \end{equation*}
        Thus, we infer
                \begin{equation*}
                    \left(\frac{\mu_2(k)}{\mu_2(n)}\right)^{{\alpha}}\leqslant M\hat{M}\left(\frac{\omega(k)}{\omega(n)}\right)^{\walpha},\quad \forall\,k\geqslant n,
        \end{equation*}
        {\it i.e.}, $\omega\ll\mu_2$.

            \item [ii)] Choose $\alpha,\walpha<0$ and assume $\mu_1\succdot\mu_2$. Then, there exist $M=M(\walpha)\geqslant 1$ and $\hat{\alpha}=\hat{\alpha}(\walpha)<0$ such that
        \begin{equation*}
                    \left(\frac{\mu_1(k)}{\mu_1(n)}\right)^{\hat{\alpha}}\leqslant M\left(\frac{\mu_2(k)}{\mu_2(n)}\right)^{{\walpha}},\quad \forall\,k\geqslant n.
        \end{equation*}
Moreover, since $\omega\gg\mu_1$, there exists $\hat{M}=\hat{M}(\hat{\alpha},\alpha)\geqslant 1$ such that
        \begin{equation*}
                    \left(\frac{\omega(k)}{\omega(n)}\right)^{{\alpha}}\leqslant \hat{M}\left(\frac{\mu_1(k)}{\mu_1(n)}\right)^{\hat{\alpha}},\quad \forall\,k\geqslant n.
        \end{equation*}
        Therefore, we infer
                \begin{equation*}
                    \left(\frac{\omega(k)}{\omega(n)}\right)^{{\alpha}}\leqslant M\hat{M}\left(\frac{\mu_2(k)}{\mu_2(n)}\right)^{\walpha},\quad \forall\,k\geqslant n,
        \end{equation*}
        {\it i.e.}, $\omega\gg\mu_2$.
    \end{itemize}
\end{proof}

The previous result gives a qualitative description of the relations $\precdot$ and $\succdot$. We have:
\begin{itemize}
    \item [i)] If $\mu_1$ is almost slower than $\mu_2$, that is $\mu_1\precdot\mu_2$, then every growth rate that is slower than $\mu_1$ will also be slower than $\mu_2$.
    \item [ii)] If $\mu_1$ is almost faster than $\mu_2$, that is $\mu_1\succdot\mu_2$, then every growth rate that is faster than $\mu_1$ will also be faster than $\mu_2$.
\end{itemize}

This description will allow us to create a reflexive relation on $\mathcal{G}$ based on $\ll$. Now, we present the two notions of equivalence that will derive in orders based on $\prec$ and $\ll$.

\begin{definition}\label{equivalences}
    Consider $\omega,\mu\in \mathcal{G}$. We say that
    \begin{itemize}
        \item [i)] $\omega$ is \textbf{weakly equivalent} to $\mu$, and denote it $\omega\sim \mu$, if $\omega\prec\mu\prec \omega$
        \item [ii)] $\omega$ is \textbf{equivalent} to $\mu$, and denote it $\omega\approx \mu$, if $\omega\precdot\mu\precdot \omega$ and $\omega\succdot\mu\succdot \omega$.
    \end{itemize}
\end{definition}

 It is immediate that $\sim$ is an equivalence relation. Consider for instance the growth rate $\mu:\mathbb{R}\to \mathbb{R}^+$ defined by
       \begin{equation*}
		\mu(t):= \left\{ \begin{array}{lcc}
		\ceg(t) &  \text{ if } &   |t|\geqslant a^*, \\
  p(t)  &\text{if}& |t|<a^*,
		\end{array}
		\right.
		\end{equation*}
        where $a^*$ is the only positive solution to the equation $\ceg(t)=p(t)$. It is immediate to verify that $\mu\sim \ceg$. Moreover, on $\mathcal{G}/\sim$, the relation $\prec$ becomes an order. 
    
Similarly, it is clear that $\approx$ also defines an equivalence relation on $\mathcal{G}$. For instance, if we define $\mu(t)=e^{3t}$, we find $\mu\approx\exp$, although $\mu\not\ll\exp$ nor $\exp\not\ll\mu$. Note that this definition also solves the problem of the non reflexivity of the relation $\ll$.

Although elements on the quotients $\mathcal{G}/\sim$ or $\mathcal{G}/\approx$ are not growth rates but classes of functions, we will call them growth rates, since at any occasion it is enough to choose a class representative.

Note that the name {\it weakly equivalent} for the relation $\sim$ comes from the relation being defined through the weak comparison, even tho it is a more stringent classification than $\approx$.  Indeed, note that $\mu_1\sim\mu_2\Rightarrow \mu_1\approx\mu_2$, but the converse is not true. For instance, if we consider $\omega\in \mathcal{G}$ and $\mu(k)=\left(\omega(k)\right)^\theta$, for some $\theta\geq 1$ as in Example \ref{810}, it is easy to check that $\omega\approx \mu$ but $\omega\not\sim \mu$,  since $\omega\prec \mu$ but $\mu\not\prec\omega$.

An immediate corollary is derived from Lemma \ref{903}, for which we do not give a proof since it is elementary.

\begin{proposition}\label{909}
Let $\mu_1,\mu_2\in \mathcal{G}$.
    \begin{align*}
    \mu_1\approx\mu_2\Longrightarrow\forall\,\omega\in \mathcal{G}:\,&\left(\left(\omega\ll\mu_1\Leftrightarrow\omega\ll\mu_2\right)\land \left(\mu_1\ll\omega\Leftrightarrow\mu_2\ll\omega\right)\right).
\end{align*}
That is, if two growth rates are equivalent, then they have the same families of growth rates that are faster and slower than them.
\end{proposition}

\begin{remark}
    {\rm
An open question, for which we have neither a proof nor a counterexample, is whether the implication in Proposition \ref{909} is actually an equivalence.
    }
\end{remark}

Now, we are in a position to define an order on $\mathcal{G}/\approx$ based on the relation $\precdot$, given by
\begin{align*}
    \mu_1\lll\mu_2\Longleftrightarrow \mu_1\precdot\mu_2\land \mu_2\succdot\mu_1.
\end{align*}
More importantly, we have the following implication in terms of the relation $\ll$, which also follows directly from Lemma \ref{903}.

\begin{corollary} Let $\mu_1,\mu_2\in\mathcal{G}$.
    \begin{align*}
    \mu_1\lll\mu_2\Longrightarrow\forall\,\omega\in \mathcal{G}:\,&\left(\left(\omega\ll\mu_1\Rightarrow\omega\ll\mu_2\right)\land \left(\mu_2\ll\omega\Rightarrow\mu_1\ll\omega\right)\right).
\end{align*}
\end{corollary}

Employing this order, we can give a more precise interpretation of the duality between bounded growth and dichotomy we mentioned on Remark \ref{907}. This result is immediately derived from Theorems \ref{805} and \ref{806}.

\begin{theorem}\label{811}
    Consider a totally ordered ($\lll$) chain $\mathcal{C}$ on $\mathcal{G}/\approx$. For every system \eqref{700}, there is at most one growth rate $\mu\in \mathcal{C}$ for which the system admits both $\mu$-bounded growth and $\mu$-dichotomy.
\end{theorem}

  As the classification given by $\sim$ is more stringent than $\approx$, Corollary \ref{811} is not true if we consider $\mathcal{G}/\sim$ instead of $\mathcal{G}/\approx$. In some sense, the elements in $\mathcal{G}/\approx$ are unique, but not in $\mathcal{G}/\sim$.


We conclude this study with the following theorem, which establishes how these criteria of equivalence characterize the spectra obtained for a given system.
 
\begin{theorem}\label{908}
    Consider two growth rates $\mu,\omega$ and system \eqref{700}.
    \begin{itemize}
        \item [i)] If $\mu\sim\omega$, then  $\Sigma_\muD(A)=\Sigma_\omegaD(A)$.

        \medskip
        
        \item [ii)] If $\mu\approx \omega$, then $\Sigma_\muD(A)$ and $\Sigma_\omegaD(A)$ are qualitatively equivalent. That is:
        \begin{itemize}
        \item [a)]  $\pm\infty\in \Sigma_\muD(A)$ if and only if $\pm\infty\in \Sigma_\omegaD(A)$,
            \item [b)] $\rho_\muD(A)$ and $\rho_\omegaD(A)$ have the same amount of spectral gaps,
            \item [c)] there is a ordered correspondence between the spectral gaps,
            \item [d)] correspondent spectral gaps have the same associated invariant projector,
            \item [e)] correspondent spectral gaps are contained in the same real semiaxis.
        \end{itemize}
    \end{itemize}
\end{theorem}

\begin{proof}
    \begin{itemize}
        \item [i)] Choose $\gamma\in \rho_\muD(A)$. The $(\mu,\gamma)$-weighted system has $\muD$ with some parameters $(\P;\alpha,\beta)$ and constant $K$. We can find $M$ such that
   \begin{align*}
        \|\Phi_{\omega,\gamma}(k,n)\P(n)\|&=\|\Phi_{\mu,\gamma}(k,n)\P(n)\|\cdot  \left(\frac{\omega(k)}{\omega(n)}\right)^{-\gamma}\left(\frac{\mu(k)}{\mu(n)}\right)^{\gamma}\\
        &\leqslant M K\cdot \left(\frac{\mu(k)}{\mu(n)}\right)^{\alpha} \\
        &\leqslant M^2 K\cdot \left(\frac{\omega(k)}{\omega(n)}\right)^{\alpha},\quad \forall\,k\geqslant n.
    \end{align*}
    and similarly
       \begin{align*}
        \|\Phi_{\omega,\gamma}(k,n)[\Id-\P(n)]\| &\leqslant M^2 K\cdot \left(\frac{\omega(k)}{\omega(n)}\right)^{\beta},\quad \forall\,k\leqslant n,
    \end{align*}
thus $\gamma\in \rho_\omegaD(A)$. By the symmetry of the definition we obtain the conclusion.

\medskip
        
        \item [ii)] It is enough to prove that for every $\gamma\in \rho_\muD(A)$, there is some $\zeta\in \rho_\omegaD(A)$ with $\sgn(\gamma)=\sgn(\zeta)$ and such that the $(\omega,\zeta)$-weighted system has $\omegaD$ with the same projector that the $(\mu,\gamma)$-weighted system has $\muD$.

        Choose $\gamma\in \rho_\muD(A)$ and suppose without loss of generality that $\gamma>0$. So, the $(\mu,\gamma)$-weighted system has $\muD$ with some parameters $(\P;\alpha,\beta)$. We can find $\zeta>0$, $\walpha<0$, $\wbeta>0$ and $M\geqslant 1$ such that
  \begin{align*}
        \|\Phi_{\omega,\zeta}(k,n)\P(n)\|&=\|\Phi_{\mu,\gamma}(k,n)\P(n)\|\cdot  \left(\frac{\omega(k)}{\omega(n)}\right)^{-\zeta}\left(\frac{\mu(k)}{\mu(n)}\right)^{\gamma}\\
        &\leqslant KM\cdot \left(\frac{\mu(k)}{\mu(n)}\right)^{\alpha} \\
        &\leqslant K M^2\cdot\left(\frac{\omega(k)}{\omega(n)}\right)^{\walpha} ,\quad \forall\,k\geqslant n.
    \end{align*}
and similarly
  \begin{align*}
        \|\Phi_{\omega,\zeta}(k,n)[\Id-\P(n)]\|&\leqslant KM^2\cdot \left(\frac{\omega(k)}{\omega(n)}\right)^{\wbeta} ,\quad \forall\,k\leqslant n,
    \end{align*}
 which completes the proof.
    \end{itemize}
\end{proof}

\end{document}